\documentclass{amsart}
\usepackage{xcolor}
\usepackage{hyperref}
\hypersetup{
	colorlinks=true,
	citecolor=blue,
	linkcolor=blue,
	filecolor=magenta,      
	urlcolor=cyan,
}
\usepackage[capitalise,noabbrev]{cleveref}
\usepackage[all, knot]{xy}\SelectTips{eu}{}

\usepackage{amsmath,amscd}
\usepackage{amssymb}
\usepackage{mathtools}
\usepackage{stmaryrd}
\usepackage{url}
\usepackage{tikz-cd}
\usepackage{enumitem}

\usepackage{fullpage}

\usepackage{setspace}
\usepackage{microtype}

\usepackage{times}

\usepackage{tikz}

\usepackage{enumitem,kantlipsum}

\address{Department of Algebra, Faculty of Mathematics and Physics, Charles University in Prague, Sokolovsk\'a 83, 186 75 Praha, Czech Republic}

\email{shaul@karlin.mff.cuni.cz}

\newtheorem{thm}[equation]{Theorem}
\newtheorem*{thm*}{Theorem}
\newtheorem*{cor*}{Corollary}
\newtheorem*{dfn*}{Definition}

\newtheorem{cthm}{Theorem}
\newtheorem{ccor}[cthm]{Corollary}

\newtheorem{cor}[equation]{Corollary}
\newtheorem{prop}[equation]{Proposition}

\theoremstyle{definition}
\newtheorem{dfn}[equation]{Definition}
\newtheorem{rem}[equation]{Remark}

\newcommand{\opn}{\operatorname}
\newcommand{\cat}[1]{\operatorname{\mathsf{#1}}}

\newcommand{\mfrak}[1]{\mathfrak{#1}}

\newcommand{\mrm}[1]{\mathrm{#1}}
\newcommand{\mbb}[1]{\mathbb{#1}}

\newcommand{\K}{\mbb{K} \hspace{0.05em}}

\newcommand{\p}{\mfrak{p}}

\newcommand{\injdim}{\operatorname{inj\,dim}}
\newcommand{\projdim}{\operatorname{proj\,dim}}
\newcommand{\flatdim}{\operatorname{flat\,dim}}

\def\skewtimes{\ltimes\!}

\newcommand{\op}{\opn{op}}

\title{Acyclic complexes of injectives and finitistic dimensions}
\author{Liran Shaul \\ \\ with an appendix by Tsutomu Nakamura and Peder Thompson}

\begin{document}
\begin{abstract}
For a ring $A$, 
we consider the question whether
every bounded above cochain complex of injective $A$-modules which is acyclic is null-homotopic.
We show that if $A$ is left and right noetherian and has a dualizing complex, 
then this implies that the finitistic dimension of $A$ is finite.
In the appendix, Nakamura and Thompson show that the opposite holds over any ring.
Our results give several new necessary and sufficient conditions for a ring to have finite finitistic dimension in a very general setting. Applications include a generalization of a recent result of Rickard about relations between unbounded derived categories and finitistic dimension,
as well as several new characterizations of noetherian rings which satisfy the Gorenstein symmetry conjecture.
\end{abstract}

\numberwithin{equation}{section}
\maketitle

\setcounter{section}{-1}
\section{Introduction}

In this paper rings are associative and unital,
modules are by default left modules,
and complexes are indexed cohomologically.
Given a ring $A$, 
our goal is to gain a better understanding of homological numerical invariants associated to it.
Specifically, we are interested in the (big) projective, injective and flat finitistic dimensions of $A$.
The big projective finitistic dimension of $A$ is given by 
\[
\opn{FPD}(A) := \sup\{\projdim_A(M)\mid M\in \opn{Mod}(A), \projdim_A(M) < \infty\},
\]
and the injective and flat finitistic dimensions, 
denoted by $\opn{FID}(A)$ and $\opn{FFD}(A)$ are defined similarly.
The question whether these invariants are finite for artin algebras is known as the big finitistic dimension conjecture,
and is considered one of the most important open problems in homological algebra. See the survey \cite{Zim} for history of this problem.

One of the most basic facts in homological algebra,
which is used in the construction of derived functors,
is the fact that over any ring $A$, every bounded below cochain complex of injective $A$-modules which is acyclic is null-homotopic.
Equivalently, any quasi-isomorphism between two bounded below complexes of injective $A$-modules is a homotopy equivalence.
This basic fact implies the uniqueness, up to a homotopy equivalence, of an injective resolution of a given module.
On the projective side, over any ring, 
every bounded above cochain complex of projective $A$-modules which is acyclic is null-homotopic.

In this paper we consider statements dual to these.
That is, we consider the question,
given a ring $A$, whether every bounded above complex of injective $A$-modules which is acyclic is null-homotopic.
Similarly, we ask, whether every bounded below complex of projective $A$-modules which is acyclic is null-homotopic.

Our main result shows that for left and right noetherian rings with a dualizing complex, these questions are equivalent to 
the question of whether finitistic dimensions are finite.
The notion of a dualizing complex, first introduced by Grothendieck \cite{RD} in algebraic geometry, and by Yekutieli \cite{Yek} in non-commutative algebra is essential in our proof techniques. Examples show that in the absence of it, our main results do not hold.
Most noetherian rings that arise in nature,
and in particular, finite dimensional algebras over a field possess dualizing complexes. 

To state our main results, we first recall the notion of a localizing (respectively colocalizing) subcategory of the unbounded derived category $\cat{D}(A)$,
where $A$ is a ring.
Given a non-empty set $\mathbf{S}$ of objects in $\cat{D}(A)$,
the localizing (resp. colocalizing) subcategory generated by $\mathbf{S}$,
denoted by $\opn{Loc}(\mathbf{S})$ (resp. $\opn{Coloc}(\mathbf{S})$),
is by the definition the smallest full triangulated subcategory of $\cat{D}(A)$ which contains $\mathbf{S}$ and is closed under infinite coproducts (resp. products). We will only be interested in the cases where $\mathbf{S}$ is either $\opn{Proj}(A)$, the collection of all projective $A$-modules, $\opn{Inj}(A)$, the collection of all injective $A$-modules, or $\opn{Flat}(A)$, the collection of all flat $A$-modules.

As our first main result, we show in \cref{thm:acyclic}:
\begin{cthm}\label{cthmA}
Let $A$ be a left and right noetherian ring which has a dualizing complex.
Assume that $\opn{Coloc}(\opn{Proj}(A)) = \cat{D}(A)$.
Then every bounded above cochain complex of injective $A$-modules which is acyclic is null-homotopic.
\end{cthm}

Specializing \cref{cthmA} to commutative rings,
we obtain in \cref{cor:commutative} the following unconditional result:

\begin{ccor}
Let $A$ be a commutative noetherian ring which has a dualizing complex.
Then every bounded above cochain complex of injective $A$-modules which is acyclic is null-homotopic.
\end{ccor}

The dual result about bounded below complexes of projective $A$-modules also holds. See \cref{rem:dual} for details.
In the appendix, Nakamura and Thompson further generalize this result, showing in \cref{cor:acyc-inj-general} that it holds over any commutative noetherian ring.

Our next results concern the big injective finitistic dimension. Recall that it is given by 
\[
\opn{FID}(A) = \sup\{\injdim_A(M) \mid M \in \opn{Mod}(A), \injdim_A(M) < \infty\}.
\]
Our first result concerning the finitistic dimension is given in \cref{thm:FPD} and it states:
\begin{cthm}\label{cthmB}
Let $A$ be a left and right noetherian ring with a dualizing complex,
and suppose that every bounded above complex of injective $A$-modules which is acyclic is null-homotopic.
Then $\opn{FID}(A) < \infty$.
\end{cthm}

Combining \cref{cthmA} and \cref{cthmB}, 
we obtain the following, which generalizes \cite[Proposition 5.2]{Rickard} where Rickard proved the same result for finite dimensional algebras over a field:

\begin{ccor}
Let $A$ be a left and right noetherian ring with a dualizing complex.
If $\opn{Coloc}(\opn{Proj}(A)) = \cat{D}(A)$ then $\opn{FID}(A) < \infty$.
\end{ccor}

In the appendix, Nakamura and Thompson show that for any ring, 
the converse of \cref{cthmB} (and the dual statement about finitistic projective dimension) hold. Using this, our results above, and our results from \cite{ShFinDim}, 
we obtain the following which is the main result of this paper,
repeated below as \cref{thm:main}:

\begin{cthm}\label{cthm:main}
Let $A$ be any ring, and consider the following statements that $A$ might satisfy:
\begin{enumerate}
\item The injective left $A$-modules generate $\cat{D}(A)$, that is $\opn{Loc}(\opn{Inj}(A)) = \cat{D}(A)$.
\item The flat right $A$-modules cogenerate $\cat{D}(A^{\op})$, that is $\opn{Coloc}(\opn{Flat}(A^{\op})) = \cat{D}(A^{\op})$.
\item The projective right $A$-modules cogenerate $\cat{D}(A^{\op})$, that is $\opn{Coloc}(\opn{Proj}(A^{\op})) = \cat{D}(A^{\op})$.
\item Every bounded above complex of injective right $A$-modules which is acyclic is null-homotopic, that is, $\cat{K}^{-}_{\mrm{ac}}(\opn{Inj}(A^{\op})) := \{M \in \cat{K}^{-}(\opn{Inj}(A^{\op})) \mid \forall n, \mrm{H}^n(M) = 0\} = 0$.
\item Every bounded below complex of projective left $A$-modules which is acyclic is null-homotopic, that is, $\cat{K}^{+}_{\mrm{ac}}(\opn{Proj}(A)) := \{ M \in \cat{K}^{+}(\opn{Proj}(A)) \mid \forall n, \mrm{H}^n(M) = 0\} = 0$.
\item It holds that $\opn{FPD}(A) < \infty$.
\item It holds that $\opn{FFD}(A) < \infty$.
\item It holds that $\opn{FID}(A^{\op}) < \infty$.
\end{enumerate}
Then the following holds:
\begin{enumerate}
\item[a)] It always holds that  (8) $\implies$ (4), that (7) $\impliedby$ (6) $\implies$ (5) $\impliedby$ (7) and that (3) $\implies$ (2).
\item[b)] If moreover $A$ is left coherent then (1) $\implies$ (2),
and if $A$ is left noetherian then (7) $\iff$ (8).
\item[c)] If $A$ is left and right noetherian and has a dualizing complex, 
then 
\[
(1)\implies (2) \iff (3) \implies (4)  \iff (5)  \iff  (6) \iff (7) \iff (8)
\]
\item[d)] If $A$ does not have a dualizing complex, even if it is commutative and noetherian,
it could happen that (1), (2), (4) and (5) hold, but (6), (7) and (8) do not hold.
\end{enumerate}
\end{cthm}

It is particularly interesting to note that the above implies that when $A$ is left and right noetherian, and has a dualizing complex, 
then the question if $\opn{FPD}(A) < \infty$ (which specializes to the finitistic dimension conjecture in case $A$ is an artin algebra) is thus equivalent to statements (4) and (5) above,
while if $A$ does not have a dualizing complex, this is not true.

Using the above and the covariant Grothendieck duality,
we give in \cref{thm:exact-criteria} another necessary and sufficient condition for a left and right noetherian ring with a dualizing complex $R$ to have $\opn{FPD}(A) < \infty$. It turns out that finite finitistic dimension is equivalent to the functor $R\otimes_A -$ being exact on the category of bounded below cochain complexes of projectives.

Restricting our attention to finite dimensional algebras over a field,
and using a recent symmetry result of Cummings \cite{Cummings}, 
we obtain in \cref{cor:symmetry} the following equivalent form of the finitistic dimension conjecture:

\begin{ccor}
Let $\K$ be a field.
Then every finite dimensional algebra over $\K$ satisfies that $\opn{FPD}(A) < \infty$ if and only if for every finite dimensional algebra $A$ over $\K$ the following are equivalent:
\begin{enumerate}
\item Every bounded above complex of injective left $A$-modules which is acyclic is null-homotopic.
\item Every bounded below complex of projective left $A$-modules which is acyclic is null-homotopic.
\end{enumerate}
\end{ccor}

We then show in \cref{thm:limit} that if a left and right noetherian ring $A$ with a dualizing complex has finite small finitistic dimension, then $\opn{FPD}(A) < \infty$ if and only if every acyclic bounded below complex of projectives is a direct limit of acyclic bounded below complexes of finitely generated projective modules.

As another application of the above, we apply the above results to study Gorenstein symmetry.
The Gorenstein symmetry conjecture asks, 
given a left and right noetherian ring $A$,
with the property that $\injdim_A(A) < \infty$,
whether it holds that $\injdim_{A^{\op}}(A^{\op})<\infty$.
In that case, by \cite[Lemma A]{Zaks}, it holds that $\injdim_A(A) = \injdim_{A^{\op}}(A^{\op})$.
Most of the existing literature on this problem focuses on the case where $A$ is an artin algebra,
but the problem makes sense (and is open)
in the much more general noetherian case considered here.
As far as we know, 
no counterexamples are known for the problem over noetherian rings.
We wish to provide new criteria for a noetherian ring $A$ to satisfy the Gorenstein symmetry conjecture. An important feature of our main result is that we are able to, assuming $A$ has finite injective dimension as a left module over itself, obtain several necessary and sufficient conditions for $A$ to have finite injective dimension as a right module over itself in terms of the behavior of left $A$-modules and complexes of left $A$-modules.

We say that an $A$-$A$ bimodule $\phantom{ }_A M_A$ has a \textbf{bi-injective resolution} if there is a bounded below complex of bimodules $\phantom{ }_A I_A$ which is quasi-isomorphic to $M$, and such that forgetting its bimodule structure, 
it is a complex of injectives both on the left and on the right. For example, if there exists a commutative ring $\K$ such that $A$ is a flat central $\K$-algebra,
then every $A$-$A$ bimodule has a bi-injective resolution.

Our main result about Gorenstein symmetry, given in \cref{thm:Gorenstein} states:
\begin{cthm}
Let $A$ be a ring which is left and right noetherian,
and such that the regular bimodule $\phantom{ }_A A_A$ has a bi-injective resolution.
Assume that $\injdim_A(A) < \infty$.
Then the following are equivalent:
\begin{enumerate}
\item $A$ satisfies the Gorenstein symmetry conjecture: It holds that $\injdim_{A^{\op}}(A^{\op}) < \infty$.
\item $A$ has a dualizing complex and $\opn{Coloc}(\opn{Proj}(A)) = \cat{D}(A)$.
\item $A$ has a dualizing complex and $\opn{Coloc}(\opn{Flat}(A)) = \cat{D}(A)$.
\item $A$ has a dualizing complex, and every bounded above complex of injective $A$-modules which is acyclic is null-homotopic.
\item $A$ has a dualizing complex, and every injective $A$-module has finite flat dimension.
\item $A$ has a dualizing complex $R$ which has finite flat dimension as a complex of left $A$-modules.
\end{enumerate}
\end{cthm}

The fact that conditions (5) and (6) are equivalent to (1) is not new, and is of course well known.
It can be easily seen that for any ring, (5)$\implies$(3),
but condition (3) is in general much weaker than (5).
For instance, any commutative noetherian ring satisfies that 
$\opn{Coloc}(\opn{Flat}(A)) = \cat{D}(A)$,
but this is far from true for (5):
the only commutative noetherian rings that satisfy (5) are the Gorenstein rings of finite Krull dimension,
so that one may view (3) as a substantial weakening of the classical condition for Gorenstein symmetry.
For rings with dualizing complexes, condition (4) is even weaker. As discussed in \cref{rem:flat} below,
for any left and right noetherian ring with a dualizing complex,
it holds that (3)$\implies$(4),
so our main result identifies much weaker necessary and sufficient conditions for Gorenstein symmetry to hold.

\textbf{Acknowledgments.}

The author thanks Jan \v S\v tov\'\i\v cek, Amnon Yekutieli, Tsutomu Nakamura and Peder Thompson for helpful discussions.
This work has been supported by the grant GA~\v{C}R 20-02760Y from the Czech Science Foundation.

\section{Preliminaries}

In this section we gather various preliminaries used throughout the paper.

\subsection{Derived categories and DG-rings}

We will freely use the language of derived categories over rings and DG-rings. 
We refer the reader to \cite{Kel,YeBook} for background about these.
By a DG-ring $A$ we will always mean a non-positive DG-ring $A = \bigoplus_{n=-\infty}^0 A^n$, equipped with differential of degree $+1$.
For a ring or a DG-ring $A$, 
we denote by $A^{\op}$ the opposite ring or DG-ring.
As modules and DG-modules are always left DG-modules,
right DG-modules will be considered as DG-modules over $A^{\op}$.
The derived category of complexes of left $A$-modules (or left DG-modules in cases $A$ is a DG-ring) will be denoted by $\cat{D}(A)$.
The bounded derived category $\cat{D}^{\mrm{b}}(A)$ is the full triangulated subcategory consisting of all $M$ such that $\mrm{H}^n(M) = 0$ for all $|n|\gg 0$.
Given $M \in \cat{D}(A)$, 
we let $\inf(M) = \inf\{n \in \mathbb{Z} \mid \mrm{H}^n(M) \ne 0\}$,
and similarly $\sup(M) = \sup\{n \in \mathbb{Z} \mid \mrm{H}^n(M) \ne 0\}$.

For a DG-ring $A$, we have that $\mrm{H}^0(A)$ is a ring,
and there is a natural map of DG-rings $A \to \mrm{H}^0(A)$.
This induces two functors 
\[
\mrm{H}^0(A)\otimes^{\mrm{L}}_A -, \mrm{R}\opn{Hom}_A(\mrm{H}^0(A),-) :\cat{D}(A) \to \cat{D}(\mrm{H}^0(A)),
\]
known as the reduction and coreduction functors,
which are extremely useful in the study of DG-rings.
They were studied in detail in \cite{ShINJ,ShRed,YeDual},
and we will frequently use their properties.

\subsection{Homological and finitistic dimensions}

If $A$ is either a ring or a DG-ring, 
we will consider various homological dimensions of complexes or DG-modules over $A$.
These are defined for a given $M \in \cat{D}^{\mrm{b}}(A)$ as follows:
The projective dimension of $M$ is given by
\[
\projdim_A(M) = \inf\{n \in \mathbb{Z} \mid \opn{Ext}_A^i(M,N) = 0 \text{ for all $N \in \cat{D}^\mrm{b}(A)$ and any $i > n + \sup(N)$}\},
\]
the injective dimension of $M$ is defined as
\[
\injdim_A(M) = \inf\{n \in \mathbb{Z} \mid \opn{Ext}_A^i(N,M) = 0 \text{ for all $N \in \cat{D}^\mrm{b}(A)$ and any $i > n - \inf(N)$}
\},
\]
and the flat dimension of $M$ is
\[
\flatdim_A(M) = \inf\{n \in \mathbb{Z} \mid \opn{Tor}^A_i(N,M) = 0\text{ for all $N \in \cat{D}^\mrm{b}(A^{\op})$ and any $i > n - \inf(N)$}\}.
\]

We shall further require the notions of finitistic dimensions over rings and DG-rings.
Over a ring $A$, we define the (big) projective finitistic dimension by 
$\opn{FPD}(A) = \sup\{\projdim_A(M) \mid M \in \opn{Mod}(A), \projdim_A(M) < \infty\}$
Similarly, the injective finitistic dimension of $A$ is 
$\opn{FID}(A) = \sup\{\injdim_A(M) \mid M \in \opn{Mod}(A), \injdim_A(M) < \infty\}$
and the flat finitistic dimension of $A$ is
$\opn{FFD}(A) = \sup\{\flatdim_A(M) \mid M \in \opn{Mod}(A), \flatdim_A(M) < \infty\}$
For our purposes however, 
these are not sufficient, 
and we shall also require the corresponding notions over DG-rings with bounded cohomology, following \cite{DGFinite,ShFinDim}.
It should be noted that in that context, 
where modules are replaced by DG-modules, 
we must normalize the definitions.
Thus, the projective finitistic dimension of a DG-ring $A$ is given by
\[
\opn{FPD}(A) = \sup\{\projdim_A(M) + \inf(M) \mid M \in \cat{D}^{\mrm{b}}(A), \projdim_A(M) < \infty\}.
\]
The injective finitistic dimension is
\[
\opn{FID}(A) = \sup\{\injdim_A(M) - \sup(M) \mid M \in \cat{D}^{\mrm{b}}(A), \injdim_A(M) < \infty\},
\]
and the flat finitistic dimension is given by the formula
\[
\opn{FFD}(A) = \sup\{\flatdim_A(M) + \inf(M) \mid M \in \cat{D}^{\mrm{b}}(A), \flatdim_A(M) < \infty\}.
\]

\subsection{Homotopy categories}

For a ring $A$, recall that $\cat{K}(A)$ is the homotopy category of complexes of $A$-modules.
We let $\opn{Proj}(A)$, $\opn{Inj}(A)$ and $\opn{Flat}(A)$ denote the collections of all projective, injective and flat $A$-modules.
We shall need the following three full triangulated subcategories of the homotopy category: $\cat{K}(\opn{Inj}(A))$, $\cat{K}(\opn{Proj}(A))$ and $\cat{K}(\opn{Flat}(A))$.
By definition, $\cat{K}(\opn{Inj}(A))$ is the homotopy category of all complexes $I$ such that $I^n$ is an injective $A$-module for all $n$. The categories $\cat{K}(\opn{Proj}(A))$ and $\cat{K}(\opn{Flat}(A))$ are defined similarly.
For more information about these triangulated categories,
see \cite{JorCom,Kra05,Nee08,Stov}.
One important fact we will frequently use is that by \cite[Proposition 8.1]{Nee08},
for any ring $A$, the inclusion functor $\cat{K}(\opn{Proj}(A)) \to \cat{K}(\opn{Flat}(A))$ has a right adjoint.
Following \cite{InKr}, we will denote this right adjoint by $q:\cat{K}(\opn{Flat}(A)) \to \cat{K}(\opn{Proj}(A))$.

\subsection{Complexes of bimodules and bi-injective resolutions}

Given a ring $A$, 
we will often work with complexes $\phantom{ }_A M_A$ of $A$-$A$-bimodules. 
For such a complex $M$,
we define its restriction $\opn{Rest}_A(M) \in \cat{D}(A)$ (respectively $\opn{Rest}_{A^{\op}}(M) \in \cat{D}(A^{\op})$) to be the complex obtained from $M$ by forgetting its right (resp. left) structure.

\begin{dfn}
Let $A$ be a ring, 
and let $\phantom{ }_A M_A$ be a bounded below complex of $A$-$A$-bimodules. 
A bi-injective resolution of $M$ is an $A$-$A$-linear quasi-isomorphism $\phantom{ }_A M_A \to \phantom{ }_A I_A$,
such that $\phantom{ }_A I_A$ is a complex of $A$-$A$-bimodules which satisfies:
\begin{enumerate}
\item $I^n = 0$ for all $n \ll 0$.
\item For each $n$,
the left $A$-module $\left(\opn{Rest}_A(I)\right)^n$ is an injective left $A$-module.
\item For each $n$,
the right $A$-module $\left(\opn{Rest}_{A^{\op}}(I)\right)^n$ is an injective right $A$-module.
\end{enumerate}
\end{dfn}

Often any bounded below complex of bimodules has a bi-injective resolution,
as in the next result, 
which is essentially contained in \cite[Proposition 2.4]{Yek}:
\begin{prop}
Let $A$ be a ring.
Suppose that there exist a commutative ring $\K$,
and a map $\K \to A$, 
whose image lies in the center of $A$,
and such that $A$ is flat over $\K$.
Then every bounded below complex $M$ of $A$-$A$-bimodules has a bi-injective resolution.
\end{prop}
\begin{proof}
The point is that the commutative ring $\K$ allows us to form the enveloping algebra $A\otimes_{\K} A^{\op}$,
and a complex of $A$-$A$-bimodules is simply a complex over $A\otimes_{\K} A^{\op}$.
Flatness then ensures that for any injective $A\otimes_{\K} A^{\op}$-module $I$,
the restrictions $\opn{Rest}_A(I)$ and $\opn{Rest}_{A^{\op}}(I)$ are injective over $A$ and $A^{\op}$,
which implies that any injective resolution of $M$ over $A\otimes_{\K} A^{\op}$ is a bi-injective resolution of $M$ over $A$.
\end{proof}

\subsection{Dualizing complexes and the covariant Grothendieck duality}

One of the key tools used in this paper is the theory of dualizing complexes over noncommutative rings,
first introduced by Yekutieli in \cite{Yek},
and studied in detail in \cite{InKr,Pos,vdb,YZ2}.

For our purposes, following \cite{Pos},
we wish to distinguish between weak dualizing complexes and strong dualizing complexes.
We start with recalling the former:

\begin{dfn}\label{dfn:weak}
Let $A$ be a left and right noetherian ring.
A complex of bimodules $\phantom{ }_A R_A$ is called a \textbf{weak dualizing complex} over $A$ if the following holds:
\begin{enumerate}
\item It holds that  $\injdim_A(\opn{Rest}_A(R)) < \infty$ and $\injdim_{A^{\op}}(\opn{Rest}_{A^{\op}}(R)) < \infty$.
\item The cohomology of $R$ is bounded and finitely generated on both sides: for each $n$, $\mrm{H}^n(\opn{Rest}_A(R))$ is a finitely generated $A$-module and $\mrm{H}^n(\opn{Rest}_{A^{\op}}(R))$ is a finitely generated $A^{\op}$-module.
\item The natural homothety maps $A \to \mrm{R}\opn{Hom}_A(R,R)$ and $A \to \mrm{R}\opn{Hom}_{A^{\op}}(R,R)$ are isomorphisms.
\end{enumerate}
\end{dfn}

As a key example, 
if $A$ is a left and right noetherian ring such that 
$\injdim_A(A) = \injdim_{A^{\op}}(A^{\op}) < \infty$,
then by definition, the regular bimodule $\phantom{ }_A A_A$ is a weak dualizing complex over itself.

We will however be focusing mostly on strong dualizing complexes, which we will simply call dualizing complexes.

\begin{dfn}
Let $A$ be a left and right noetherian ring.
A complex of bimodules $\phantom{ }_A R_A$ is called a \textbf{dualizing complex} over $A$ if the following holds:
\begin{enumerate}
\item $R$ is a weak dualizing complex.
\item $R$ is bounded, so $R^n = 0$ for all $|n| \gg 0$.
\item For all $n$, $\left(\opn{Rest}_A(R)\right)^n$ is an injective $A$-module and $\left(\opn{Rest}_{A^{\op}}(R)\right)^n$ is an injective $A^{\op}$-module.
\end{enumerate}
\end{dfn}

Thus, a dualizing complex is simply a bounded complex of bimodules which is simultaneously a complex of injectives on the left and on the right, satisfying the axioms of a weak dualizing complex.

\begin{prop}\label{prop:bi-injective}
Let $A$ be a left and right noetherian ring,
and let $R$ be a weak dualizing complex over $A$.
If $R$ has a bi-injective resolution,
then $A$ has a dualizing complex.
\end{prop}
\begin{proof}
Let $S$ be a bi-injective resolution of $R$.
It is possible that $S$ is not bounded,
but the fact that $\injdim_A(\opn{Rest}_A(S)) < \infty$ and $\injdim_{A^{\op}}(\opn{Rest}_{A^{\op}}(S)) < \infty$,
implies, 
as in the proof of \cite[Proposition I.7.6]{RD},
that for $n$ large enough, 
the smart truncation $\sigma_{\le n}(S)$ of $S$ below $n$ will be a bounded complex of $A$-$A$-bimodules which is injective on both sides, and which is quasi-isomorphic to $S$.
Hence, it is a dualizing complex over $A$.
\end{proof}

The reason we must work with dualizing complexes rather than weak dualizing complexes, 
is that they induce the covariant Grothendieck duality:
by \cite[Theorem 4.8]{InKr} (see also \cite[Theorem 4.5]{Pos}),
if $R$ is a dualizing complex over a left and right noetherian ring $A$, then the functor $R\otimes_A -:\cat{K}(\opn{Proj}(A)) \to \cat{K}(\opn{Inj}(A))$ is an equivalence of categories,
with the inverse equivalence given by $q\circ \opn{Hom}_A(R,-)$,
where $q:\cat{K}(\opn{Flat}(A)) \to \cat{K}(\opn{Proj}(A))$ is the right adjoint to the inclusion.

\subsection{Localizing and colocalizing subcategories}

Given a triangulated category $\mathcal{T}$,
a localizing subcategory of $\mathcal{T}$ is a full triangulated subcategory of $\mathcal{T}$ which is closed under infinite coproducts. For a non-empty subset $\mathbf{S}$ of objects of $\mathcal{T}$,
we denote by $\opn{Loc}(\mathbf{S})$ the localizing subcategory of $\mathcal{T}$ generated by $\mathbf{S}$,
which is, by definition, the smallest localizing subcategory of $\mathcal{T}$ which contains $\mathbf{S}$.
If $\opn{Loc}(\mathbf{S}) = \mathcal{T}$,
we will say that $\mathbf{S}$ generates $\mathcal{T}$.

Similarly, a colocalizing subcategory of $\mathcal{T}$ is a full triangulated subcategory of $\mathcal{T}$ which is closed under infinite products. 
Again, given a non-empty subset $\mathbf{S}$ of objects of $\mathcal{T}$,
we let $\opn{Coloc}(\mathbf{S})$ be the colocalizing subcategory of $\mathcal{T}$ cogenerated by $\mathbf{S}$,
that is, the smallest colocalizing subcategory of $\mathcal{T}$ which contains $\mathbf{S}$. 
If $\opn{Coloc}(\mathbf{S}) = \mathcal{T}$,
we say that $\mathbf{S}$ cogenerates $\mathcal{T}$.

For any ring $A$, by \cite[Proposition 2.2]{Rickard},
it holds that $\opn{Loc}(\opn{Proj}(A)) = \cat{D}(A)$.
Similarly, one shows that $\opn{Coloc}(\opn{Inj}(A)) = \cat{D}(A)$. A bit more generally, it is not hard to see that if $I$ is any injective cogenerator of $\opn{Mod}(A)$,
then the singleton $\{I\}$ cogenerates $\cat{D}(A)$.

Following ideas of Keller, it was demonstrated in \cite{Rickard} that it is interesting and fruitful to also consider the categories $\opn{Loc}(\opn{Inj}(A))$ and $\opn{Coloc}(\opn{Proj}(A))$. In this paper we will also consider $\opn{Coloc}(\opn{Flat}(A))$.

The next result is taken from our upcoming joint paper with Peder Thompson \cite{ShThompson}.
We thank Thompson for allowing us to include its proof here.
See \cite[Section 4.2]{BW} for similar results.

\begin{thm}\label{thm:injImplyFlat}
Let $A$ be a left coherent ring.
If $\mrm{Loc}(\opn{Inj}(A)) = \cat{D}(A)$ then
\[
\mrm{Coloc}(\opn{Flat}(A^{\opn{op}})) = \cat{D}(A^{\opn{op}}).
\]
\end{thm}
\begin{proof}
Define the following full subcategory of $\cat{D}(A)$:
\[
\mathcal{S} := \{M \in \cat{D}(A) \mid \opn{Hom}_{\mathbb{Z}}(M,\mathbb{Q}/\mathbb{Z}) \in \mrm{Coloc}(\opn{Flat}(A^{\opn{op}}))\}.
\]
It follows from the definition that $\mathcal{S}$ is a localizing subcategory of $\cat{D}(A)$.
As $A$ is left coherent,
according to \cite[Lemma 3.1.4]{Xu},
if $I$ is a left injective $A$-module,
it holds that $\opn{Hom}_{\mathbb{Z}}(I,\mathbb{Q}/\mathbb{Z})$ is a flat $A^{\opn{op}}$-module,
which implies that $I \in \mathcal{S}$.
Hence, the fact that $\mrm{Loc}(\opn{Inj}(A)) = \cat{D}(A)$ implies that $\mathcal{S} = \cat{D}(A)$.
This implies that $\opn{Hom}_{\mathbb{Z}}(A,\mathbb{Q}/\mathbb{Z}) \in \mrm{Coloc}(\opn{Flat}(A^{\opn{op}}))$,
and since it is an injective cogenerator of $\cat{D}(A^{\opn{op}})$,
we deduce that 
$\mrm{Coloc}(\opn{Flat}(A^{\opn{op}})) = \cat{D}(A^{\opn{op}})$.
\end{proof}

\section{Projective cogeneration and acyclic complexes of injectives}

In this section we relate the property that the projective modules cogenerate the derived category over a noetherian ring with a dualizing complex to the contractiblity of bounded above acyclic complexes of injectives.
Our main result is the following:

\begin{thm}\label{thm:acyclic}
Let $A$ be a left and right noetherian ring which has a dualizing complex. 
Assume that $\opn{Coloc}(\opn{Proj}(A)) = \cat{D}(A)$.
Then every bounded above cochain complex of injective $A$-modules which is acyclic is null-homotopic.
\end{thm}
\begin{proof}
Let $R$ be a dualizing complex over $A$,
and suppose that there exist a bounded above complex of injective $A$-modules $D$ which is acyclic but not null-homotopic. As it is a complex of injectives,
we may view $D$ as an element of $\cat{K}(\opn{Inj}(A))$.
Take some projective left $A$-module $P$,
and note that by the Bass-Papp theorem,
the complex $R\otimes_A P$ is a bounded complex of injectives left $A$-modules.
Since $D$ is acyclic,
when considered as an element of $\cat{D}(A)$,
we have that $D \cong 0$,
which implies that
\[
\opn{Hom}_{\cat{D}(A)}(D,R\otimes_A P[n]) = 0
\]
for all $n \in \mathbb{Z}$.
Since $R\otimes_A P[n]$ is a bounded complex of injectives,
it is K-injective, 
so by \cite[Theorem 10.1.13]{YeBook},
it follows that
\[
0 = \opn{Hom}_{\cat{D}(A)}(D,R\otimes_A P[n]) = 
\opn{Hom}_{\cat{K}(A)}(D,R\otimes_A P[n]).
\]
As both sides are complexes of injectives,
we see that
\[
0 = \opn{Hom}_{\cat{K}(A)}(D,R\otimes_A P[n]) = 
\opn{Hom}_{\cat{K}(\opn{Inj}(A))}(D,R\otimes_A P[n]).
\]
According to \cite[Theorem 4.8]{InKr} (or \cite[Theorem 4.5]{Pos}), the functor 
\[
q\circ\opn{Hom}_{A}(R,-):\cat{K}(\opn{Inj}(A)) \to \cat{K}(\opn{Proj}(A))
\]
is an equivalence of categories. 
Applying this equivalence gives us 
\begin{equation}\label{eqn:keyEq}
0 = \opn{Hom}_{\cat{K}(\opn{Proj}(A))}(q\circ \opn{Hom}_{A}(R,D),q\circ \opn{Hom}_{A}(R,R\otimes_A P[n])).
\end{equation}
By definition, $\opn{Hom}_{A}(R,D)$ is a bounded above complex of flat $A$-modules.
According to \cite[Theorem 2.7(2)]{InKr},
this implies that
the complex $q\circ \opn{Hom}_{A}(R,D)$ is a bounded above complex of projective $A$-modules.
Next, let us analyze the right hand size, namely
$q\circ\opn{Hom}_{A}(R,R\otimes_A P[n])$.
To do this, we first consider its image in $\cat{D}(A)$.
There, we may calculate that
\begin{gather*}
q\circ \opn{Hom}_{A}(R,R\otimes_A P[n]) \cong 
\mrm{R}\opn{Hom}_{A}(R,R\otimes^{\mrm{L}}_A P[n]) \cong \\
\mrm{R}\opn{Hom}_{A}(R,R) \otimes^{\mrm{L}}_A P[n] \cong P[n]
\end{gather*}
where the first isomorphism follows from \cite[Theorem 2.7(1)]{InKr}, the second isomorphism follows from \cite[Theorem 12.9.10]{YeBook} and the last isomorphism follows from the fact that $R$ is a dualizing complex.
An alternative way to see this is to observe that we just applied the two mutually inverse equivalences $R\otimes_A -$ and $q\circ \opn{Hom}_{A}(R,-)$ to the element $P[n]\in \cat{K}(\opn{Proj}(A))$, 
which explains why the result is $P[n]$.

Going back to (\ref{eqn:keyEq}),
we have seen that $q\circ \opn{Hom}_{A}(R,D)$ is a bounded above complex of projective $A$-modules,
which implies that it is K-projective.
Hence, we deduce using \cite[Theorem 10.2.9]{YeBook} that
\begin{gather*}
0 = \opn{Hom}_{\cat{K}(\opn{Proj}(A))}(q\circ \opn{Hom}_{A}(R,D),q\circ \opn{Hom}_{A}(R,R\otimes_A P[n])) =\\
\opn{Hom}_{\cat{K}(A)}(q\circ \opn{Hom}_{A}(R,D),q\circ \opn{Hom}_{A}(R,R\otimes_A P[n])) =\\
\opn{Hom}_{\cat{D}(A)}(q\circ \opn{Hom}_{A}(R,D),q\circ \opn{Hom}_{A}(R,R\otimes_A P[n])) \cong\\
\opn{Hom}_{\cat{D}(A)}(q\circ \opn{Hom}_{A}(R,D),P[n]).
\end{gather*}
We may now finish the proof as follows.
Consider the full subcategory of $\cat{D}(A)$ given by
\[
\mathcal{S} = \{X \in \cat{D}(A) \mid \forall n, \opn{Hom}_{\cat{D}(A)}(q\circ \opn{Hom}_{A}(R,D),X[n]) = 0\}.
\]
Clearly $\mathcal{S}$ is closed under shifts, taking cones and taking products,
so it is a colocalizing subcategory.
Moreover, we have seen that $\mathcal{S}$ contains all projective $A$-modules,
so we deduce that 
 $\opn{Coloc}(\opn{Proj}(A)) \subseteq \mathcal{S}$.
 Since $q\circ \opn{Hom}_{A}(R,-)$ is an equivalence of categories, 
and since $D \ncong 0$ in $\cat{K}(\opn{Inj}(A))$,
we deduce that $q\circ \opn{Hom}_{A}(R,D) \ncong 0$ in  $\cat{K}(\opn{Proj}(A))$.
Since it is K-projective, it follows that 
$q\circ \opn{Hom}_{A}(R,D) \ncong 0$ in $\cat{D}(A)$.
 We deduce that the identity map 
 \[
 q\circ \opn{Hom}_{A}(R,D) \to q\circ \opn{Hom}_{A}(R,D)
 \]
 is not the zero map.
 This shows that
 \[
 \opn{Hom}_{\cat{D}(A)}(q\circ \opn{Hom}_{A}(R,D),q\circ \opn{Hom}_{A}(R,D)) \ne 0,
 \]
 so that $q\circ \opn{Hom}_{A}(R,D) \notin \mathcal{S}$.
 Hence, $\opn{Coloc}(\opn{Proj}(A)) \subseteq \mathcal{S} \subsetneq \cat{D}(A)$
as claimed.
\end{proof}

\begin{rem}\label{rem:flat}
Over a left and right noetherian ring with a dualizing complex, it was shown in \cite[Theorem]{Jor05} that every flat module has finite projective dimension,
which implies that over such a ring there is an equality
\[
\opn{Coloc}(\opn{Proj}(A)) = \opn{Coloc}(\opn{Flat}(A)),
\]
so that one may replace projective modules by flat modules in the statement of the above result.
\end{rem}

\begin{rem}\label{rem:dual}
The dual result to \cref{thm:acyclic}, namely, that if $A$ is a left and right noetherian ring with a dualizing complex, 
such that $\opn{Loc}(\opn{Inj}(A)) = \cat{D}(A)$,
then every bounded below cochain complex of projective $A$-modules which is acyclic is null-homotopic also holds.
The proof of this fact is essentially contained in the proof of \cite[Theorem 5.1]{ShFinDim},
but we did not realize this at the time of writing that proof.
\end{rem}

We now restrict our attention to commutative rings.
In that case we obtain the following stronger unconditional result:
\begin{cor}\label{cor:commutative}
Let $A$ be a commutative noetherian ring which has a dualizing complex.
Then every bounded above cochain complex of injective $A$-modules which is acyclic is null-homotopic.    
\end{cor}
\begin{proof}
In view of \cref{thm:acyclic} we only need to show that $\opn{Coloc}(\opn{Proj}(A)) = \cat{D}(A)$,
or equivalently, by \cref{rem:flat},
that $\opn{Coloc}(\opn{Flat}(A)) = \cat{D}(A)$.
According to \cite[Theorem 3.3]{Rickard},
for any commutative noetherian ring $A$ it holds that $\opn{Loc}(\opn{Inj}(A)) = \cat{D}(A)$,
so the result follows from \cref{thm:injImplyFlat} using the fact that $A = A^{\op}$.
Alternatively, one may see directly that for any commutative noetherian ring it holds that $\opn{Coloc}(\opn{Flat}(A)) = \cat{D}(A)$, 
by using Neeman's classification \cite{Neeman2011} of colocalizing subcategories over commutative noetherian rings,
noticing that the flat $A$-module given by the $\p$-adic completion of the localization $A_{\p}$ has cosupport, 
in the sense of \cite{BIK},
equal to $\{\p\}$ for any $\p \in \opn{Spec}(A)$.
\end{proof}

\begin{cor}
Let $A$ be a commutative noetherian ring which has a dualizing complex,
and let $I,J$ be two bounded above cochain complexes of injective $A$-modules.
Then every quasi-isomorphism $\varphi:I \to J$ is a homotopy equivalence.
\end{cor}
\begin{proof}
Letting $D$ be the cone of $\varphi$,
the fact that $I$ and $J$ are both bounded above cochain complexes of injective $A$-modules implies that $D$ is also a bounded above cochain complex of injective $A$-modules.
The fact that $\varphi$ is a quasi-isomorphism implies that $D$ is acyclic, 
so by \cref{cor:commutative}, it is null-homotopic.
This in turn implies that $\varphi$ is a homotopy equivalence.
\end{proof}

\section{Injective finitistic dimension of Gorenstein DG-rings}

This section is dedicated to establishing the finiteness of the injective finitistic dimension of a Gorenstein DG-ring, 
under a further mild conditions that will hold in all of our applications.
This result will be used in the next section to connect the acyclicity of bounded above complexes of injectives with the projective finitistic dimension of a noetherian ring with a dualizing complex.

\begin{prop}\label{prop:flatProjdim}
Let $A$ be a non-positive DG-ring with bounded cohomology.
Suppose that every flat $\mrm{H}^0(A)$-module has finite projective dimension over $\mrm{H}^0(A)$.
Then every bounded DG-module $M \in \cat{D}^{\mrm{b}}(A)$ such that $\flatdim_A(M) < \infty$ satisfies $\projdim_A(M) < \infty$.
\end{prop}
\begin{proof}
By shifting if necessary, 
we may assume that $\sup(M) = 0$.
We prove the result by induction on $\flatdim_A(M)$. 
The fact that $\sup(M) = 0$ implies by the definition of flat dimension of DG-modules that $\flatdim_A(M) \ge 0$.
Assume first that $\flatdim_A(M) = 0$.
Let $\bar{M} = \mrm{H}^0(A)\otimes^{\mrm{L}}_A M$.
By \cite[proof of Proposition 3.1]{YeDual}, we have that $\sup(\bar{M}) = 0$,
and by \cite[Corollary 1.5(ii)]{DGFinite} it holds that $\flatdim_{\mrm{H}^0(A)}(\bar{M}) = 0$.
Moreover, the definition of flat dimension and the fact that $\inf(\mrm{H}^0(A)) = \sup(\mrm{H}^0(A)) = 0$ implies that $\inf(\bar{M}) = 0$.
We deduce that $\bar{M}$ is a flat $\mrm{H}^0(A)$-module,
so by assumption it has finite projective dimension over $\mrm{H}^0(A)$.
According to \cite[Corollary 1.5(i)]{DGFinite}, we know that $\projdim_A(M) = \projdim_{\mrm{H}^0(A)}(\bar{M})$,
establishing the claim.
Next, suppose that $\flatdim_A(M) > 0$.
By using \cite[Lemma 2.8]{Min},
we can find a DG-module of the form $F = \bigoplus_{j \in J} A$,
and a map of DG-modules $\varphi:F \to M$ such that $\mrm{H}^0(\varphi)$ is surjective.
Such a construction makes $\varphi$ the first step of a spft resolution of $M$,
in the sense of \cite[Definition 4.9]{Min}.
Embedding $\varphi$ inside a distinguished triangle
\begin{equation}\label{eqn:distForFlat}
N \to F \xrightarrow{\varphi} M \to N[1]
\end{equation}
in $\cat{D}(A)$,
since $\sup(F) = \sup(M) = 0$,
the fact that $\mrm{H}^0(\varphi)$ is surjective implies that $\sup(N) \le 0$.
By \cite[Lemma 4.10]{Min}, we have that $\flatdim_A(N) < \flatdim_A(M)$.
Hence, by the induction hypothesis, we know that $\projdim_A(N) < \infty$.
As $\projdim_A(F) = 0 <\infty$,
the distinguished triangle (\ref{eqn:distForFlat}) implies that $\projdim_A(M) < \infty$.
\end{proof}

The following fact is obvious for rings, 
but for DG-rings it might require a short proof:
\begin{prop}\label{prop:ineqflatproj}
Let $A$ be a non-positive DG-ring with bounded cohomology,
and let $M \in \cat{D}^{\mrm{b}}(A)$.
Then 
\[
\flatdim_A(M) \le \projdim_A(M).
\]
\end{prop}
\begin{proof}
We may assume that $\projdim_A(M) < \infty$,
for otherwise there is nothing to prove.
Let $\bar{M} = \mrm{H}^0(A)\otimes^{\mrm{L}}_A M \in \cat{D}(\mrm{H}^0(A))$.
We know that $\sup(\bar{M}) = \sup(M) < \infty$,
and by \cite[Corollary 1.5(i)]{DGFinite} we have that $\projdim_{\mrm{H}^0(A)}(\bar{M}) = \projdim_A(M) < \infty$,
so in particular $\bar{M}$ is a bounded complex.
Since flat and projective dimension over rings can be detected from the length of flat and projective resolutions, clearly $\flatdim_{\mrm{H}^0(A)}(\bar{M}) \le \projdim_{\mrm{H}^0(A)}(\bar{M})$.
Finally, by \cite[Corollary 1.5(ii)]{DGFinite} we have that $\flatdim_A(M) = \flatdim_{\mrm{H}^0(A)}(\bar{M})$,
so we obtain
\[
\flatdim_A(M) = \flatdim_{\mrm{H}^0(A)}(\bar{M}) \le \projdim_{\mrm{H}^0(A)}(\bar{M}) = \projdim_A(M).
\]
\end{proof}

\begin{prop}\label{prop:FFDifFPD}
Let $A$ be a non-positive DG-ring with bounded cohomology.
Suppose that every flat $\mrm{H}^0(A)$-module has finite projective dimension over $\mrm{H}^0(A)$.
If $\opn{FPD}(A) < \infty$ then $\opn{FFD}(A) \le \opn{FPD}(A) < \infty$.
\end{prop}
\begin{proof}
Set $n = \opn{FPD}(A) < \infty$, and let $M \in \cat{D}^{\mrm{b}}(A)$ be a DG-module with $\flatdim_A(M) < \infty$.   
By \cref{prop:flatProjdim}, we know that $\projdim_A(M) < \infty$.
Hence, by the definition of $\opn{FPD}(A)$, we deduce that $\projdim_A(M) + \inf(M) \le n$.
From \cref{prop:ineqflatproj} we obtain the inequality
\[
\flatdim_A(M) + \inf(M) \le \projdim_A(M) + \inf(M) \le n.
\]
The definition of $\opn{FFD}(A)$ now implies that 
\[
\opn{FFD}(A) \le n = \opn{FPD}(A) < \infty.
\]
\end{proof}

We will say that a DG-ring $A$ is left noetherian if the ring $\mrm{H}^0(A)$ is left noetherian and for all $i<0$,
the left $\mrm{H}^0(A)$-module $\mrm{H}^i(A)$ is finitely generated. Similarly for right noetherian. 
This definition is justified by \cite[Theorem 6.6]{ShINJ}.

\begin{prop}\label{prop:ffdFID}
Let $A$ be a left and right noetherian DG-ring with bounded cohomology.
Then $\opn{FFD}(A) = \opn{FID}(A^{\op})$.
\end{prop}
\begin{proof}
For rings, this is a well known result due to Matlis \cite[Theorem 1]{Mat59}.
For commutative DG-rings, this was shown in \cite[Corollary 6.16]{DGFinite},
and the same proof carries over to this noncommutative situation.
\end{proof}

Here is the main result of this section.
\begin{thm}\label{thm:FID}
Let $A$ be a left and right noetherian DG-ring with bounded cohomology,
and suppose that every flat $\mrm{H}^0(A)$-module has finite projective dimension over $\mrm{H}^0(A)$.
If $\injdim_A(A) < \infty$ then $\opn{FID}(A^{\op}) < \infty$.
\end{thm}
\begin{proof}
According to \cite[Theorem 3.5]{ShFinDim},
the fact that $\injdim_A(A) < \infty$ implies that $\opn{FPD}(A) < \infty$.
By \cref{prop:FFDifFPD} this implies that $\opn{FFD}(A) < \infty$,
and from \cref{prop:ffdFID} we get that $\opn{FID}(A^{\op}) < \infty$.
\end{proof}

Given a ring $A$, and a complex of $A$-$A$-bimodules $\phantom{ }_A M_A$ with the property that $\sup(M) < 0$,
the trivial extension DG-ring $A\skewtimes M$ was defined in \cite[Definition 4.1]{ShFinDim}.
It is a non-positive DG-ring equipped with a map of DG-rings $\tau_{A,M}: A\to A\skewtimes M$,
such that as a complex of bimodules over $A$ it is isomorphic to $A\oplus M$,
such that $\mrm{H}^0(A\skewtimes M) = A$, and the composition $A \xrightarrow{\tau_{A,M}} A\skewtimes M \to \mrm{H}^0(A\skewtimes M) = A$ is equal to $1_A$.

\begin{cor}\label{cor:InfFID}
Let $A$ be a left and right noetherian ring,
and let $R$ be a dualizing complex over $A$ such that $\sup(R) < 0$.
Then the DG-ring $B = A\skewtimes R$ satisfies $\opn{FID}(B) < \infty$ and $\opn{FID}(B^{\op}) < \infty$.
\end{cor}
\begin{proof}
By \cite[Theorem]{Jor05}, 
the fact that $A$ has a dualizing complex implies that every flat $A$-module has finite projective dimension.
According to \cite[Theorem 4.5]{ShFinDim}, 
the DG-ring $B$ satisfies $\injdim_B(B) < \infty$ and $\injdim_{B^{\op}}(B^{\op}) < \infty$.
Since $\mrm{H}^0(B) = A$,
the result now follows from \cref{thm:FID}.
\end{proof}

\section{Acyclic complexes of injectives and finitistic dimension}

The aim of this section is to relate properties of acyclic complexes of injective modules to the injective finitistic dimension of a noetherian ring with a dualizing complex.

We first prove the following implication of a ring with a dualizing complex having infinite finitistic dimension,
which is a dual of \cite[Theorem 4.12]{ShFinDim}. 
This will then be used to construct a certain acyclic complex of injectives in such a situation.

\begin{thm}\label{thm:seperation}
Let $A$ be a left and right noetherian ring with a dualizing complex $S$.
Suppose that $\opn{FID}(A) = +\infty$.
Then there is an integer $j \in \mathbb{Z}$,
an increasing sequence of natural numbers $a_n$,
with $\lim_{n\to +\infty} a_n = +\infty$,
and for each $n$,
a left $A$-module $M_n$,
with $\injdim_A(M_n) = a_n$,
and $\opn{Ext}^{a_n}_A(R,M_n) \ne 0$,
where $R = S[j]$.
\end{thm}
\begin{proof}
The condition that $\opn{FID}(A) = +\infty$ implies the existence of an increasing sequence of natural numbers $b_n$,
with $\lim_{n\to +\infty} b_n = +\infty$,
and for each $n$,
a left $A$-module $M_n$,
with $\injdim_{A}(M_n) = b_n$.
By shifting if necessary, 
we may assume that $\sup(S) < 0$.
Let  $B = A\skewtimes S$ be the trivial extension DG-ring of $A$ by $S$.
According to \cref{cor:InfFID}, 
it holds that $K = \opn{FID}(B) < \infty$.
For each $n$, let $N_n = \mrm{R}\opn{Hom}_A(B,M_n) \in \cat{D}(B)$.
The fact that the composition $A \to B \to \mrm{H}^0(B) = A$ is equal to the identity implies by adjunction that
\[
\mrm{R}\opn{Hom}_B(\mrm{H}^0(B),N_n) = \mrm{R}\opn{Hom}_B(A,N_n) \cong M_n.
\]
According to \cite[Theorem 2.5]{ShINJ},
there is an equality 
\[
\injdim_B(N_n) = \injdim_{\mrm{H}^0(B)}\left(\mrm{R}\opn{Hom}_B(\mrm{H}^0(B),N_n)\right) = \injdim_A(M_n) = b_n < \infty.
\]
Hence, by the definition of the injective finitistic dimension,
we have an inequality
\[
K = \opn{FID}(B) \ge \injdim_B(N_n) - \sup(N_n) = b_n - \sup(N_n).
\]
The definition of injective dimension over $A$ implies that
\[
\sup(N_n) = \sup\left(\mrm{R}\opn{Hom}_A(B,M_n)\right) \le b_n - \inf(B) = b_n - \inf(S).
\]
Combining the above, 
we see that for each $n$ there are inequalities
\[
-K \le \sup(N_n) - b_n \le -\inf(S).
\]
As $K,\inf(S)$ are fixed constants,
this implies the existence of some $j \in \mathbb{Z}$,
such that $\sup(N_n) - b_n = j$ for infinitely many $n$.
Focusing on these $n$, we obtain an increasing sequence $a_n$ of natural numbers, a subsequence of $b_n$,
such that $\injdim_A(M_n) = a_n$, 
and such that $\sup(N_n) = a_n + j$.
The fact that over $A$ it holds that $B \cong A\oplus S$
implies that
\[
\sup(N_n) = \sup\left(\mrm{R}\opn{Hom}_A(B,M_n)\right) = 
\sup\left(\mrm{R}\opn{Hom}_A(A\oplus S,M_n)\right) = 
\sup\left(\mrm{R}\opn{Hom}_A(S,M_n)\right).
\]
As $\mrm{R}\opn{Hom}_A(-,M_n)$ is a contravariant triangulated functor,
this implies that
\[
\sup\left(\mrm{R}\opn{Hom}_A(S[j],M_n)\right) = 
\sup\left(\mrm{R}\opn{Hom}_A(S,M_n)\right) - j =
\sup(N_n) - j = a_n,
\]
which shows that
\[
0 \ne \mrm{H}^{a_n}\left(\mrm{R}\opn{Hom}_A(S[j],M_n)\right) = \opn{Ext}^{a_n}_A(S[j],M_n) 
\]
for all $n$, as claimed.
\end{proof}

We now use this non-vanishing result to prove the main result of this section which relates acyclic complexes of injectives to the finitistic dimension.

\begin{thm}\label{thm:FPD}
Let $A$ be a left and right noetherian ring with a dualizing complex.
Assume that every bounded above complex of injective left $A$-modules which is acyclic is null-homotopic.
Then $\opn{FID}(A) < \infty$.
\end{thm}
\begin{proof}
Suppose that the conclusion is false,
so that $\opn{FID}(A) = +\infty$.
Then \cref{thm:seperation} holds,
which implies the existence of $a_n, M_n$ and $R$ as in the conclusion of that result.
Note that $R$, 
being a shift of a dualizing complex is also a dualizing complex.
For each $n$,
let $I_n$ be an injective resolution of $M_n$ of length $a_n$.
Consider the natural map
\[
\varphi: \bigoplus_{n=1}^{\infty} I_n[a_n] \to \prod_{n=1}^{\infty} I_n[a_n]
\]
Since $\mrm{H}^{-a_n}(I_n[a_n]) = M_n$,
and $\mrm{H}^i(I_n[a_n]) = 0$ for all $i\ne -a_n$,
it follows that $\varphi$ is a quasi-isomorphism.

Next, we claim that $\varphi$ is not a homotopy equivalence.
To see this, 
recall that the restriction of $R$ to $\cat{D}(A)$, 
the complex of left $A$-modules $\opn{Rest}_A(R)$ has bounded and finitely generated cohomology.
Hence, by \cite[Proposition 7.4.9]{YeBook},
there is an isomorphism $\opn{Rest}_A(R) \cong S$ in $\cat{D}(A)$, 
where $S$ is a bounded complex of finitely generated $A$-modules. 
This in turn implies that the functor $\opn{Hom}_A(S,-)$ commutes with both products and coproducts.
Applying this functor to $\varphi$,
and using the fact that 
\[
\mrm{H}^0\left(\opn{Hom}_A(S,I_n[a_n])\right) = \opn{Ext}^{a_n}_A(S,M_n) \cong \opn{Ext}^{a_n}_A(R,M_n),
\]
we deduce that $\mrm{H}^0\left(\opn{Hom}_A(S,\varphi)\right)$ is isomorphic to the natural map
\[
\bigoplus_{n=1}^{\infty}\opn{Ext}^{a_n}_A(R,M_n) \to \prod_{n=1}^{\infty} \opn{Ext}^{a_n}_A(R,M_n).
\]
Since for all $n$, 
it holds that $\opn{Ext}^{a_n}_A(R,M_n) \ne 0$,
we deduce that the map $\mrm{H}^0\left(\opn{Hom}_A(S,\varphi)\right)$ is not an isomorphism.
As $\opn{Hom}_A(S,-)$ is an additive functor,
it preserves homotopy equivalences,
which shows that $\varphi$ is not a homotopy equivalence.

The Bass-Papp theorem implies that $\varphi$ is a map between bounded above complexes of injective $A$-modules.
This shows that its cone $D = \opn{cone}(\varphi)$ is also bounded above complex of injective $A$-modules.
Since $\varphi$ is a quasi-isomorphism,
it follows that $D$ is acyclic,
and since $\varphi$ is not homotopy equivalence,
it follows that $D$ is not null-homotopic,
contradicting our assumption on $A$.
\end{proof}

Combining \cref{thm:acyclic} and \cref{thm:FPD}, 
we have the following result:
\begin{cor}
Let $A$ be a left and right noetherian ring with a dualizing complex.
If $\opn{Coloc}(\opn{Proj}(A)) = \cat{D}(A)$,
then $\opn{FID}(A) < \infty$.
\end{cor}

\begin{rem}
In \cite[Theorem 4.3]{Rickard},
Rickard showed that if $A$ is a finite dimensional algebra over a field and $\opn{Loc}(\opn{Inj}(A)) = \cat{D}(A)$ then $\opn{FPD}(A) < \infty$.
That result was generalized by us in \cite[Corollary 5.3]{ShFinDim}, where the same result was shown to hold for every left and right noetherian ring with a dualizing complex.
It should be noted that every finite dimensional algebra over a field has a dualizing complex.
Similarly, Rickard showed in \cite[Proposition 5.2]{Rickard} that if $A$ is a finite dimensional algebra over a field and $\opn{Coloc}(\opn{Proj}(A)) = \cat{D}(A)$,
then $\opn{FPD}(A^{\op}) < \infty$,
and so the above result completes the picture and generalizes to the same setting Rickard's dual result.
Some ideas of Rickard's proof of \cite[Proposition 5.2]{Rickard} are contained in the proofs of \cref{thm:acyclic} and \cref{thm:FPD},
and one should view his result as a predecessor of both of these theorems.
\end{rem}

\section{Equivalent conditions for finite finitistic dimension}

In this section we prove several equivalent conditions for a ring to have finite finitistic dimensions. We start with the main result of this paper, \cref{cthm:main} from the introduction.
Let us recall its statement:

\begin{thm}\label{thm:main}
Let $A$ be any ring, and consider the following statements that $A$ might satisfy:
\begin{enumerate}
\item The injective left $A$-modules generate $\cat{D}(A)$, that is $\opn{Loc}(\opn{Inj}(A)) = \cat{D}(A)$.
\item The flat right $A$-modules cogenerate $\cat{D}(A^{\op})$, that is $\opn{Coloc}(\opn{Flat}(A^{\op})) = \cat{D}(A^{\op})$.
\item The projective right $A$-modules cogenerate $\cat{D}(A^{\op})$, $\opn{Coloc}(\opn{Proj}(A^{\op})) = \cat{D}(A^{\op})$.
\item Every bounded above complex of injective right $A$-modules which is acyclic is null-homotopic, that is, $\cat{K}^{-}_{\mrm{ac}}(\opn{Inj}(A^{\op})) := \{M \in \cat{K}^{-}(\opn{Inj}(A^{\op})) \mid \forall n, \mrm{H}^n(M) = 0\} = 0$.
\item Every bounded below complex of projective left $A$-modules which is acyclic is null-homotopic, that is, $\cat{K}^{+}_{\mrm{ac}}(\opn{Proj}(A)) := \{ M \in \cat{K}^{+}(\opn{Proj}(A)) \mid \forall n, \mrm{H}^n(M) = 0\} = 0$.
\item It holds that $\opn{FPD}(A) < \infty$.
\item It holds that $\opn{FFD}(A) < \infty$.
\item It holds that $\opn{FID}(A^{\op}) < \infty$.
\end{enumerate}
Then the following holds:
\begin{enumerate}
\item[a)] It always holds that  (8) $\implies$ (4), that (7) $\impliedby$ (6) $\implies$ (5) $\impliedby$ (7) and that (3) $\implies$ (2).
\item[b)] If moreover $A$ is left coherent then (1) $\implies$ (2),
and if $A$ is left noetherian then (7) $\iff$ (8).
\item[c)] If $A$ is left and right noetherian and has a dualizing complex, 
then 
\[
(1)\implies (2) \iff (3) \implies (4)  \iff (5)  \iff  (6) \iff (7) \iff (8)
\]
\item[d)] If $A$ does not have a dualizing complex, even if it is commutative and noetherian,
it could happen that (1), (2), (4) and (5) hold, but (6), (7) and (8) do not hold.
\end{enumerate}
\end{thm}
\begin{proof}
\begin{enumerate}[wide, labelwidth=!, labelindent=0pt]
\item[a)]The fact that (8)$\implies$(4) is \cref{FID_fin_nh}. That (6)$\implies$ (7) is a classical result of Jensen which 
follows from \cite[Proposition 6]{Jensen}. That (6) $\implies (5)$ and that (7) $\implies$ (5) is \cref{thm_proj}, and the fact that (3)$\implies$ (2) is trivial because any projective module is flat. 
\item[b)] That (1)$\implies(2)$ is \cref{thm:injImplyFlat}, and if $A$ is left noetherian then (7) $\iff$ (8) is a classical theorem of Matlis shown in \cite[Theorem 1]{Mat59}.
\item[c)] Assume now that $A$ is left and right noetherian and has a dualizing complex.
It follows from \cite[Theorem]{Jor05} that (2)$\iff$(3), and that (6) $\iff$ (7).
According to \cref{thm:FPD} it holds that (4)$\implies$(8),
and by \cref{thm:acyclic} we have that (3) $\implies$ (4).
The proof of \cite[Theorem 5.1]{ShFinDim} shows that (5)$\implies$ (6).
More precisely, on the first part of the proof,
assuming $\opn{FPD}(A) = +\infty$, 
we construct a bounded below complex of projective left $A$-modules which is acyclic but not null-homotopic.
We have thus shown that 
\[
(1) \implies (2) \iff (3) \implies (4) \iff (8) \iff (7) \iff (6) \iff (5),
\]
which completes the proof.
\item[d)] According to \cite[Theorem 3.3]{Rickard},
if $A$ is any commutative noetherian ring, 
then it satisfies (1), and hence also (2).
By \cref{cor:acyc-inj-general} and \cref{cor:comm-proj}, it also satisfies (4) and (5).
But if $A$ has infinite Krull dimension,
then all its finitistic dimensions are infinite, 
so that (6), (7) and (8) fail.
\end{enumerate}
\end{proof}

For a ring $A$, let us denote by $\cat{C}(\opn{Mod}(A))$ the category of all cochain complexes of left $A$-modules, and by
$\cat{C}^{+}(\opn{Proj}(A))$ its full subcategory consisting of 
cochain complexes $M$ such that $M^n$ is a projective $A$-module for all $n$, and $M^n = 0$ for all $n\ll 0$.
Using the above and the covariant Grothendieck duality,
were obtain the following necessary and sufficient condition for such a ring to have $\opn{FPD}(A) < \infty$:

\begin{thm}\label{thm:exact-criteria}
Let $A$ be a left and right noetherian ring with a dualizing complex $R$.
Then the following are equivalent:
\begin{enumerate}
\item It holds that $\opn{FPD}(A) < \infty$.
\item The functor $R\otimes_A -: \cat{C}^{+}(\opn{Proj}(A)) \to \cat{C}(\opn{Mod}(A))$ is exact; that is, it sends quasi-isomorphisms between objects of $\cat{C}^{+}(\opn{Proj}(A))$ to quasi-isomorphisms.
\end{enumerate}
\end{thm}
\begin{proof}
 Since the category $\cat{C}^{+}(\opn{Proj}(A))$ is closed under taking cones,
 clearly the second condition is equivalent to the fact that $R\otimes_A -: \cat{C}^{+}(\opn{Proj}(A)) \to \cat{C}(\opn{Mod}(A))$ sends acyclic complexes to acyclic complexes.
 If $\opn{FPD}(A) < \infty$,
 then by \cref{thm:main},
 every acyclic complex in $\cat{C}^{+}(\opn{Proj}(A))$ is null-homotopic, 
 and since $R\otimes_A -$ is an additive functor, 
 it sends null-homotopic complexes to null-homotopic complexes,
 proving the claim.
 Conversely, suppose that the functor $R\otimes_A -$ is exact,
 and let $M \in \cat{C}^{+}(\opn{Proj}(A))$ be an acyclic bounded below complex of projective $A$-modules.
 By assumption, $R\otimes_A M$ is also acyclic,
 but by definition it is a bounded below complex of injective $A$-modules, so the fact that it is acyclic implies that it is null-homotopic.
 According to \cite[Theorem 4.8]{InKr},
 the functor $R\otimes_A -:\cat{K}(\opn{Proj}(A)) \to \cat{K}(\opn{Inj}(A))$ is an equivalence of categories,
 and we have seen that applying it to $M$ yields $0$ in $\cat{K}(\opn{Inj}(A))$,
 which implies that $M \cong 0$ in $\cat{K}(\opn{Proj}(A))$,
 so that $M$ is null-homotopic. 
 We have thus proven that any bounded below cochain complex of projective $A$-modules which is acyclic is null-homotopic, 
 so by \cref{thm:main}, 
 we deduce that $\opn{FPD}(A) < \infty$.
\end{proof}

We now specialize to finite dimensional algebras over a field.
In this setting, Cummings has recently shown in \cite{Cummings} that if finiteness of the finitistic dimension is left-right symmetric, then any finite dimensional algebra has finite finitistic dimension. 
Combining this with \cref{thm:main}, 
we obtain the following:

\begin{cor}\label{cor:symmetry}
Let $\K$ be a field.
Then every finite dimensional algebra $A$ over $\K$ satisfies that $\opn{FPD}(A) < \infty$ if and only if for every finite dimensional algebra $A$ over $\K$ the following are equivalent:
\begin{enumerate}
\item Every bounded above complex of injective left $A$-modules which is acyclic is null-homotopic.
\item Every bounded below complex of projective left $A$-modules which is acyclic is null-homotopic.
\end{enumerate}    
\end{cor}
\begin{proof}
In view of \cref{thm:main},
for a finite dimensional algebra $A$
the fact that every bounded above complex of injective left $A$-modules which is acyclic is null-homotopic is equivalent to $\opn{FPD}(A^{\op}) < \infty$.
Similarly, the fact that every bounded below complex of projective left $A$-modules which is acyclic is null-homotopic is equivalent to $\opn{FPD}(A) < \infty$.
The result then follows from \cite[Theorem 3.4]{Cummings}.
\end{proof}

\begin{rem}
In \cite{Krause22},
Krause generalized the above result of Cummings,
constructing from any ring, 
an asymmetric cover which increases
the finitistic dimension on one side and decreases it to
zero on the other side.
It is however not clear to us if his construction,
when applied to a left and right noetherian ring with a dualizing complex, will result in a left and right noetherian ring with a dualizing complex, and hence it is not clear if one can generalize the above corollary to a larger class of rings.
\end{rem}

We finish this section by connecting finiteness of the small finitistic dimension with finiteness of the big finitistic dimension.
Recall that for a left coherent ring $A$,
we define the small projective finitistic dimension by
\[
\opn{fpd}(A) = \sup\{\projdim_A(M) \mid M \in \opn{Mod}_f(A), \projdim_A(M) < \infty\}
\]
where $\opn{Mod}_f(A)$ denotes the category of finitely presented $A$-modules.

Clearly, for any ring the fact that $\opn{FPD}(A) < \infty$ implies that $\opn{fpd}(A) < \infty$.
The next result says when, for a noetherian ring with a dualizing complex,
$\opn{fpd}(A) < \infty$ implies $\opn{FPD}(A) < \infty$.

\begin{thm}\label{thm:limit}
Let $A$ be a left and right noetherian ring with a dualizing complex, and suppose that $\opn{fpd}(A) < \infty$.
Then the following are equivalent:
\begin{enumerate}
\item It holds that $\opn{FPD}(A) < \infty$.
\item Every acyclic bounded below complex of projective $A$-modules is a filtered direct limit of acyclic bounded below complexes of finitely generated projective $A$-modules.
\end{enumerate}
\end{thm}
\begin{proof}
(1) $\implies$ (2): if $\opn{FPD}(A) < \infty$,
we know by \cref{thm_proj} that every acyclic bounded below complex $M$ of projective $A$-modules is null-homotopic.
According to \cite[Lemma 8.5]{Nee08} (and its proof),
one obtains that $M$ is a filtered direct limit of acyclic bounded below complexes of finitely generated projective $A$-modules.
(2) $\implies$ (1): by \cref{cor:fpd-small}, 
we know that any acyclic bounded below complex of finitely generated projective $A$-modules is null-homotopic.
The assumption thus implies that any 
acyclic bounded below complex $M$ of projective $A$-modules is a filtered direct limit of null-homotopic complexes of projective
$A$-modules. By \cite[Theorem 8.6]{Nee08},
this implies that $M$ itself is null-homotopic,
so by \cref{thm:main}, we deduce that $\opn{FPD}(A) < \infty$.
\end{proof}

\section{Gorenstein symmetry}

In this section we return to the Gorenstein symmetry problem discussed in the introduction.
We start with the following fact,
which is probably known,
and give a proof of it using the covariant Grothendieck duality.

\begin{prop}\label{prop:flatdimOfDual}
Let $A$ be a left and right noetherian ring with a dualizing complex.
Then the following are equivalent:
\begin{enumerate}
\item Every injective $A$-module has finite projective dimension.
\item Every injective $A$-module has finite flat dimension.
\item There is a dualizing complex $R$ over $A$ such that $\flatdim_A(R) < \infty$.
\end{enumerate}
\end{prop}
\begin{proof}
Since $A$ has a dualizing complex,
by \cite[Theorem]{Jor05}, every $A$-module of finite flat dimension has finite projective dimension, 
which shows that (1)$\iff$(2).
Let $R$ be a dualizing complex over $A$.
Since by definition it has finite injective dimension over $A$,
it is clear that (2)$\implies$(3).
To finish the proof, we must show that (3)$\implies$(2),
so suppose that $\flatdim_A(R) < \infty$,
and let $I$ be an injective $A$-module.
Applying the covariant Grothendieck duality,
let $P = q(\opn{Hom}_A(R,I)) \in \cat{K}(\opn{Proj}(A))$.
It follows from \cite[Lemma 4.1(b)]{Pos} that $\opn{Hom}_A(R,I)$ is a bounded complex of flat $A$-modules,
so using \cite[Theorem]{Jor05} again,
we deduce from \cite[Theorem 2.7]{InKr}  that $P$ is a bounded complex of projective $A$-modules.
In particular, $P$ is K-flat, 
so there is an isomorphism
\[
R\otimes^{\mrm{L}}_A P \cong R\otimes_A P,
\]
in $\cat{D}(A)$
and since the left hand side has finite flat dimension over $A$,
we deduce that $R\otimes_A P$ has finite flat dimension over $A$.
According to \cite[Theorem 4.8]{InKr},
there is an isomorphism $R\otimes_A P \cong I$ in $\cat{K}(\opn{Inj}(A))$,
so a fortiori there is such an isomorphism in $\cat{D}(A)$,
showing that $\flatdim_A(I) < \infty$.
\end{proof}

The relation between Gorensteinness and dualizing complexes is expressed in the following elementary fact:
\begin{prop}\label{prop:gorDual}
Let $A$ be a left and right noetherian ring,
such that $\injdim_A(A) = \injdim_{A^{\op}}(A^{\op}) < \infty$.
 If the regular bimodule $\phantom{ }_A A_A$ has a bi-injective resolution,
 then $A$ has a dualizing complex.
\end{prop}
\begin{proof}
Since $\injdim_A(A) = \injdim_{A^{\op}}(A^{\op}) < \infty$,
the regular bimodule $\phantom{ }_A A_A$ is a weak dualizing complex (in the sense of \cref{dfn:weak}) over itself,
so by \cref{prop:bi-injective}, 
there is a dualizing complex over $A$.
\end{proof}

We now turn to give various necessary and sufficient conditions for Gorenstein symmetry to hold.

\begin{thm}\label{thm:Gorenstein}
Let $A$ be a ring which is left and right noetherian,
and such that the regular bimodule $\phantom{ }_A A_A$ has a bi-injective resolution.
Assume that $\injdim_A(A) < \infty$.
Then the following are equivalent:
\begin{enumerate}
\item $A$ satisfies the Gorenstein symmetry conjecture: It holds that $\injdim_{A^{\op}}(A^{\op}) < \infty$.
\item $A$ has a dualizing complex and $\opn{Coloc}(\opn{Proj}(A)) = \cat{D}(A)$.
\item $A$ has a dualizing complex and $\opn{Coloc}(\opn{Flat}(A)) = \cat{D}(A)$.
\item $A$ has a dualizing complex, and every bounded above complex of injective $A$-modules which is acyclic is null-homotopic.
\item $A$ has a dualizing complex, and every injective $A$-module has finite flat dimension.
\item $A$ has a dualizing complex $R$ which has finite flat dimension as a complex of left $A$-modules.
\end{enumerate}
\end{thm}
\begin{proof}
(1) $\implies$ (3):  If $\injdim_{A^{\op}}(A^{\op}) < \infty$,
then by \cite[Theorem 3.2]{Rickard}, 
it holds that $\opn{Loc}(\opn{Inj}(A^{\op})) = \cat{D}(A^{\op})$,
so by \cref{thm:injImplyFlat} it follows that $\opn{Coloc}(\opn{Flat}(A)) = \cat{D}(A)$. Moreover,
by \cref{prop:gorDual}, $A$ has a dualizing complex.
(2) $\iff$ (3): this follows from \cite[Theorem]{Jor05},
since over a ring with a dualizing complex, a module has finite flat dimension if and only if it has finite projective dimension.
(2) $\implies$ (4): this is \cref{thm:acyclic}.
(4) $\implies$ (1): according to \cref{thm:FPD},
the assumption implies that $\opn{FPD}(A^{\op}) < \infty$,
and by \cite[Proposition 6.10]{AusRei}, 
this implies that $\injdim_{A^{\op}}(A^{\op}) < \infty$.
(5) $\iff$ (6): this is contained in \cref{prop:flatdimOfDual}.
(5) $\implies$ (3): If every injective $A$-module has finite flat dimension, then every injective $A$-module is contained in $\opn{Coloc}(\opn{Flat}(A))$,
so in particular $\opn{Coloc}(\opn{Inj}(A)) \subseteq \opn{Coloc}(\opn{Flat}(A))$, 
but for any ring it holds that $\opn{Coloc}(\opn{Inj}(A)) = \cat{D}(A)$, 
establishing the claim.
Finally, we show that (1) $\implies$ (5): Let $I$ be a left injective $A$-module,
and let $F = \opn{Hom}_{\mathbb{Z}}(I,\mathbb{Q}/\mathbb{Z})$.
By \cite[Theorem 3.2.16]{EnJen}, the right $A$-module $F$ is flat, so in particular it has finite projective dimension,
and hence, by assumption, also finite injective dimension over $A^{\op}$.
Hence, by \cite[Theorem 3.2.19]{EnJen}, 
it follows that $I$ has finite flat dimension over $A$.
\end{proof}

\begin{rem}
There is some resemblance between the fact that (4)$\iff$(5) and \cite[Corollary 5.12]{InKr},
where it is shown that for a left and right noetherian ring $A$ with a dualizing complex, the dualizing complex has finite projective dimension over $A$ if and only if every acyclic complex of injective $A$-modules is totally acyclic.
There are however several differences between the results.
First, our result is concerned only with bounded above complexes of injectives, while there the condition is on all complexes of injectives.
Secondly, our conclusion that such complexes are null-homotopic is stronger than being totally acyclic.
\end{rem}

\newcommand{\fp}{\frak{p}}
\def\FID{\operatorname{FID}}
\def\FFD{\operatorname{FFD}}
\def\FPD{\operatorname{FPD}}
\newcommand{\ZZ}{\mathbb{Z}}
\newcommand{\Z}{\operatorname{Z}}
\newcommand{\id}{\operatorname{inj dim}}
\newcommand{\Ext}{\operatorname{Ext}}
\newcommand{\fpd}{\opn{fpd}}

\appendix
\section{ \\ by Tsutomu Nakamura and Peder Thompson}

Let $A$ be a ring. The aim of this appendix is to show that if the injective, flat, or projective finitistic dimension of $A$ is finite, then certain acyclic complexes must be null-homotopic.

Let $M$ be an acyclic (cochain) complex of $A$-modules. Set $\Z^n(M)=\ker(M^n\to M^{n+1})$ for each $n\in\ZZ$, and consider the exact sequences 
\[\tag{$\dagger$}\xymatrix{
0\ar[r] & \Z^n(M)\ar[r] & M^n\ar[r] & \Z^{n+1}(M)\ar[r] & 0\;.
}\]
The complex $M$ is \emph{null-homotopic} if the identity $1_M:M\to M$ is homotopic to zero; equivalently, the sequence $(\dagger)$ splits for every $n\in\ZZ$; see, e.g., \cite[Exercise 1.4.3]{Wei94}. The complex $M$ is \emph{pure acyclic} if the exact sequences $(\dagger)$ remain exact upon application of $N\otimes_A-$ for any $A^{\operatorname{op}}$-module $N$.

\begin{thm}\label{FID_fin_nh}
Let $A$ be a ring with $\FID(A)<\infty$. Every acyclic bounded above complex of injective $A$-modules is null-homotopic. 
\end{thm}
\begin{proof}
Let $I$ be an acyclic complex of injective $A$-modules such that $I^n=0$ for $n\gg0$.  To show $I$ is null-homotopic, it suffices to show that each cocycle module $\Z^n(I)$ is injective, as then $0\to \Z^n(I) \to I^n\to \Z^{n+1}(I)\to 0$ splits for all $n\in \ZZ$. 

Set $s=\FID(A)<\infty$ and fix $n\in \ZZ$. As $I$ is bounded above, $\Z^{n-s}(I)$ has finite injective dimension, which must be at most $s$.  The module $\Z^n(I)$ is an $s$-th injective cosyzygy of $\Z^{n-s}(I)$, and thus must be injective. 
\end{proof}

The next result is essentially contained in \cite[Remark 5.7]{NT20}, but the proof given here, using Theorem \ref{FID_fin_nh}, does not explicitly rely on minimality or support theory.
\begin{cor}\label{cor:acyc-inj-general}
Let $A$ be a commutative noetherian ring. Every acyclic bounded above complex of injective $A$-modules is null-homotopic.
\end{cor}
\begin{proof}
Let $I$ be an acyclic complex of injective $A$-modules such that $I^n=0$ for $n\gg0$. For every prime ideal $\fp$ of $A$, the localization $I_\fp$ is also an acyclic bounded above complex of injective $A_\fp$-modules. Finiteness of $\FID(A_\fp)$, see \cite[Theorem 2.4]{AB58} and \cite[Theorem 1]{Mat59}, yields that $I_\fp$ is null-homotopic by Theorem \ref{FID_fin_nh}. As remarked above, this means that each cycle module $\Z^n(I)_\fp$ is injective. Thus $\Z^n(I)$ is injective as well.
\end{proof}

We turn to considering finiteness of flat and projective finitistic dimensions.

\begin{thm}\label{FFD_fin_pa}
Let $A$ be a ring with $\FFD(A)<\infty$. Every acyclic bounded below complex of flat $A$-modules is pure acyclic. 
\end{thm}
\begin{proof}
Let $F$ be an acyclic complex of flat $A$-modules such that $F^n=0$ for $n\ll0$. To show $F$ is pure acyclic, we need only show that the cycle module $\Z^n(F)$ is flat for every $n\in\ZZ$. 

Set $s=\FFD(A)$ and fix $n\in \ZZ$. As $F$ is bounded below, the module $\Z^{n+s}(F)$ has finite flat dimension, which must be at most $s$. The module $\Z^n(F)$ appears as an $s$-th flat syzygy of $\Z^{n+s}(F)$, and therefore must be flat. 
\end{proof}

\begin{cor}\label{comm_flat_pa}
Let $A$ be a commutative noetherian ring. Every acyclic bounded below complex of flat $A$-modules is pure acyclic.
\end{cor}
\begin{proof}
Let $F$ be an acyclic complex of flat $A$-modules such that $F^n=0$ for $n\ll0$. For every prime ideal $\fp$, the localization $F_\fp$ is an acyclic bounded below complex of flat $A_\fp$-modules. Finiteness of $\FFD(A_\fp)$, by \cite[Theorem 2.4]{AB58}, implies that $\Z^n(F)_\fp$ is flat for every $\fp$ by Theorem \ref{FFD_fin_pa}, and thus each $\Z^n(F)$ is flat. It follows that $F$ is pure acyclic.
\end{proof}

An $A$-module $C$ is called \emph{cotorsion} if $\Ext_A^1(F,C)=0$ holds for every flat $A$-module $F$. One then says that an $A$-module is \emph{flat cotorsion} if it is both flat and cotorsion. Every short exact sequence of flat cotorsion modules splits by definition.

\begin{cor}\label{FFD_fin_nh}
Let $A$ be a ring with $\FFD(A)<\infty$. Every acyclic bounded below complex of flat cotorsion $A$-modules is null-homotopic. 
\end{cor}
\begin{proof}
Let $F$ be an acyclic complex of flat cotorsion $A$-modules such that $F^n=0$ for $n\ll0$. Every acyclic bounded below complex of cotorsion $A$-modules has cotorsion cycle modules, so $\Z^n(F)$ is cotorsion for every $n\in\ZZ$; this is simply the coresolving property of cotorsion modules. Theorem \ref{FFD_fin_pa} implies further that $\Z^n(F)$ is flat for every $n\in\ZZ$, since pure acyclic complexes of flat modules have flat cycle modules; see e.g., \cite[Theorem 7.3]{CH15}. Hence every cycle module of $F$ is both flat and cotorsion. This yields that $F$ is null-homotopic.
\end{proof}

In the commutative noetherian case, we have the next result; cf. \cite[Lemma 5.6]{NT20}:
\begin{cor}\label{comm_fc}
Let $A$ be a commutative noetherian ring. Every acyclic bounded below complex of flat cotorsion $A$-modules is null-homotopic.
\end{cor}

\begin{proof}
This is immediate from Corollary \ref{comm_flat_pa} and the proof of Corollary \ref{FFD_fin_nh}.
\end{proof}

\begin{thm}\label{thm_proj}
Let $A$ be a ring with either $\FPD(A)<\infty$ or $\FFD(A)<\infty$. Every acyclic bounded below complex of projective $A$-modules is null-homotopic.
\end{thm}
\begin{proof}
If $\FPD(A)$ is finite, the argument is dual to that of Theorem \ref{FID_fin_nh}. If $\FFD(A)$ is finite, one employs Theorem \ref{FFD_fin_pa} along with the fact that every pure acyclic complex of projective modules (over any ring) is null-homotopic; see \cite[Remark 2.15 and Theorem 8.6]{Nee08}, and also \cite[Proposition 7.6]{CH15}.
\end{proof}

\begin{cor}\label{cor:comm-proj}
Let $A$ be a commutative noetherian ring. Every acyclic bounded below complex of projective $A$-modules is null-homotopic.
\end{cor}

\begin{proof}
This follows from Corollary \ref{comm_flat_pa} and the proof of Theorem \ref{thm_proj}.
\end{proof}

\begin{cor}\label{cor:fpd-small}
Let $A$ be a left coherent ring with $\fpd(A)<\infty$. Every acyclic bounded below complex of finitely generated projective $A$-modules is null-homotopic.
\end{cor}

\begin{proof}
As in the proof of Theorem \ref{thm_proj}, the argument is dual to that of Theorem \ref{FID_fin_nh}.
\end{proof}

\begin{rem}
If $A$ is a commutative noetherian ring (resp. $\opn{FFD}(A)$ is finite), 
then every acyclic complex of projective $A$-modules with at least one flat cycle module must be null-homotopic; this follows from \cref{comm_flat_pa} (resp. \cref{FFD_fin_pa}), 
the resolving
property of flat modules, and the proof of \cref{thm_proj}. In particular, we conclude that every flat Gorenstein projective $A$-module is projective, answering the commutative noetherian case of \cite[Question 3.8]{BCE20}; cf. the proof of \cite[Proposition 3.9]{BCE20}.    
\end{rem}

\bibliographystyle{abbrv}
\bibliography{main}

\end{document}